\documentclass[a4paper,twoside,11pt]{article}

\usepackage[margin=2cm,a4paper]{geometry} 
\usepackage{amsmath,amssymb,mathtools,xcolor}
\usepackage{latexsym}

\usepackage{amsthm, bm}
\usepackage{url, hyperref}
\usepackage{enumerate}
\usepackage{cite}
\usepackage{tikz}

\DeclarePairedDelimiter{\floor}{\lfloor}{\rfloor}
\DeclarePairedDelimiter{\ceil}{\lceil}{\rceil}
\DeclarePairedDelimiter{\card}{|}{|}

\newtheorem{thm}{Theorem}
\newtheorem{prop}[thm]{Proposition}

\newtheorem{lem}[thm]{Lemma}
\newtheorem{cor}[thm]{Corollary}
\newtheorem{conj}[thm]{Conjecture}
\numberwithin{thm}{section}

\theoremstyle{definition}
\newtheorem{dfn}[thm]{Definition}

\newtheorem{obs}[thm]{Observation}
\newtheorem{prob}{Problem}

\theoremstyle{remark}

\theoremstyle{definition}

\newcommand{\mc}[1]{\mathcal{#1}}

\def\cH{{\mathcal H}}

\def\COMMENT#1{}
\let\COMMENT=\footnote

\newcommand{\m}[1]{\mathcal{#1}}

\DeclareMathOperator{\Ryser}{Ryser}

\title{Ryser's Conjecture for $t$-intersecting hypergraphs}
\author{Anurag Bishnoi
        \thanks{Department of Applied Mathematics, Technische Universiteit Delft, Delft, Netherlands.
        E-mail: {\tt
        anurag.2357@gmail.com}. Research supported in part by a Humboldt Research Fellowship for Postdoctoral Researchers and by a Discovery Early Career Award of the Australian Research Council (No.~DE190100666).}
    \and
        Shagnik Das\thanks{Institut f\"ur Mathematik, Freie Universit\"at Berlin, 14195 Berlin, Germany.}
        \thanks{E-mail: {\tt shagnik@mi.fu-berlin.de}. Research supported in part by GIF grant G-1347-304.6/2016 and by the Deutsche Forschungsgemeinschaft (DFG) - Project 415310276.}
    \and
        Patrick Morris\footnotemark[2]
        \thanks{E-mail: {\tt pm0041@mi.fu-berlin.de}. Research supported in part by a Leverhulme Trust Study Abroad Studentship (SAS-2017-052\textbackslash 9) and by the Deutsche Forschungsgemeinschaft (DFG, German Research
Foundation) under Germany's Excellence Strategy - The Berlin Mathematics
Research Center MATH+ (EXC-2046/1, project ID: 390685689).}
    \and
        Tibor Szab\'{o}\footnotemark[2]
        \thanks{E-mail: {\tt szabo@math.fu-berlin.de}. Research supported in part by GIF grant G-1347-304.6/2016.}}

\begin{document}
\maketitle

\begin{abstract}
A well-known conjecture, often attributed to Ryser, states that the cover number of an $r$-partite $r$-uniform hypergraph is at most $r - 1$ times larger than its matching number.  Despite considerable effort, particularly in the intersecting case, this conjecture remains wide open, motivating the pursuit of variants of the original conjecture.  Recently, Bustamante and Stein and, independently, Kir\'{a}ly and T\'{o}thm\'{e}r\'{e}sz  considered the problem under the assumption that the hypergraph is $t$-intersecting, conjecturing that the cover number $\tau(\m H)$ of such a hypergraph $\m H$ is at most $r - t$.  In these papers, it was proven that the conjecture is true for $r \leq 4t-1$, but also that it need not be sharp; when $r = 5$ and $t = 2$, one has $\tau(\mc H) \le 2$.

We extend these results in two directions. First, for all $t \ge 2$ and $r \le 3t-1$, we prove a tight upper bound on the cover number of these hypergraphs, showing that they in fact satisfy $\tau(\m H) \leq \floor{(r - t)/2} + 1$. Second, we extend the range of $t$ for which the conjecture is known to be true, showing that it holds for all $r \le \frac{36}{7}t-5$. 
We also introduce several related variations on this theme.  As a consequence of our tight bounds, we resolve the problem for $k$-wise $t$-intersecting hypergraphs, for all $k \ge 3$ and $t \ge 1$.  We further give bounds on the cover numbers of strictly $t$-intersecting hypergraphs and the $s$-cover numbers of $t$-intersecting hypergraphs.
\end{abstract}

\section{Introduction}
We define an $r$-uniform hypergraph $\cH$ to be \emph{$r$-partite} if one can partition the vertex set into $r$ parts, say $V(\cH)= P_1\sqcup \dots \sqcup P_r$, such that for all $e\in E(\cH)$ and all $j\in [r]$, we have $|e\cap P_j|=1$. For $1 \leq t \leq r-1$,  an $r$-uniform hypergraph is  \emph{$t$-intersecting} if $|e\cap f|\ge t$ for all $e,f\in E(\cH)$. We will refer to $r$-uniform $r$-partite $t$-intersecting hypergraphs as \emph{$(r,t)$-graphs} throughout. 

For an $r$-uniform hypergraph $\cH$, a set of vertices $C\subset V(\cH)$ is a \emph{cover} of the hypergraph if $C\cap e \neq \emptyset$ for all $e\in E(\cH)$. The \emph{cover number} of the hypergraph, denoted $\tau(\cH)$, is the smallest cardinality of a cover for the hypergraph.  Our goal is to bound the cover numbers of $(r,t)$-graphs. 
\begin{dfn} \label{def:extremalfunction} We define the following extremal function:
\[\Ryser(r,t) = \max \, \{ \tau(\m H) : \m H \text{ is an $(r,t)$-graph} \}.\]
\end{dfn}

The problem is motivated by an old unsolved conjecture of Ryser from around 1970, first appearing in the PhD thesis of his student Henderson~\cite{Henderson71} (see~\cite{Best-Wanless18} for more on the history of this conjecture), which claims that the cover number of any $r$-uniform $r$-partite hypergraph $\m H$ satisfies $\tau(\m H) \le (r - 1) \nu(\m H)$. Here, $\nu(\m H)$ denotes the matching number of the hypergraph, that is, the size of the largest set of pairwise disjoint edges. 
When $r = 2$, this is a statement about bipartite graphs, and is equivalent to the classic theorem of K\"onig~\cite{Koenig}.  Although Ryser's Conjecture has attracted significant attention over the years, the only other resolved case is $r = 3$. This was proven via topological methods by Aharoni~\cite{Aharoni01}, with the extremal hypergraphs classified by Haxell, Narins and Szab\'o~\cite{HNS18}.

In this latter result, it was shown that the extremal hypergraphs with $\nu(\m H) = \nu \ge 2$ can essentially be decomposed into $\nu$ extremal intersecting hypergraphs. Thus, much research in this direction has focussed on intersecting hypergraphs.\footnote{Although recent constructions of Abu-Khazneh~\cite{Abu16}, for $\nu = 2$ and $r = 4$, and Bishnoi and Pepe~\cite{BP18}, for $\nu \ge 2$ and all $r \ge 4$ with $r-1$ a prime power, show that one cannot simply reduce the general case to the intersecting one.}  
Here we have $\nu(\m H) = 1$, and Ryser's Conjecture asserts $\Ryser(r,1) \le r-1$. Further motivation for considering the intersecting case comes from a connection with a conjecture of Gy\'arf\'as~\cite{gyarfasconj} which states that the vertices of any $r$-edge-coloured clique can be covered by at most $r-1$ monochromatic trees. Indeed, this conjecture in the setting of coloured complete graphs is in fact equivalent to Ryser's conjecture for intersecting hypergraphs, see e.g.~\cite{kirsolo}.
Not much more is known even in this simpler setting; Tuza~\cite{Tuza} proved the conjecture for $r \le 5$, but it remains otherwise open. The apparent difficulty of this conjecture is perhaps explained by the abundance of extremal constructions: the classic example of truncated projective planes shows $\Ryser(r,1) \ge r-1$ whenever $r-1$ is a prime power, while Abu-Khazneh, Bar\'at, Pokrovskiy and Szab\'o~\cite{ABPS19} construct exponentially (in $\sqrt{r}$) many non-isomorphic minimal examples whenever $r-2$ is a prime power. For general $r$, Haxell and Scott~\cite{HS17} construct nearly-extremal intersecting hypergraphs; more precisely, they show $\Ryser(r,1) \ge r-4$ for all $r$ large enough.

  This led Bustamante and Stein~\cite{bustastein}  and, independently, Kir\'aly and T\'othm\'er\'esz~\cite{kirtot}  to investigate what occurs when we impose the stricter condition of the hypergraph $\m H$ being $t$-intersecting. In this case, any subset of $r-t+1$ vertices from an edge must form a cover, and so we trivially have $\tau(\m H) \le r-t + 1$. While one can construct $r$-uniform $t$-intersecting hypergraphs attaining this bound, it was conjectured that, as in Ryser's Conjecture, one can do better when the hypergraph is also $r$-partite; that is, when considering $(r,t)$-graphs.

\begin{conj}[Bustamante--Stein~\cite{bustastein}, Kir\'aly--T\'othm\'er\'esz~\cite{kirtot}] \label{conj:weak}
For all $1 \le t \le r-1$, we have
\[\Ryser(r,t)\leq r-t.\]
\end{conj}

Note that while Ryser's Conjecture for intersecting hypergraphs is a special case ($t = 1$), it in fact implies Conjecture~\ref{conj:weak}. Indeed, $\Ryser(r, t) \leq \Ryser(r - t + 1, 1)$ since deleting $t - 1$ parts and removing the deleted vertices from each edge leaves us with an $(r - t + 1, 1)$-graph, which, by Ryser's Conjecture, should have a cover of size at most $r-t$. 

Therefore, one might hope to be able to make progress on Conjecture~\ref{conj:weak} for larger values of $t$, and indeed, results have been obtained when $t$ is linear in $r$.  Bustamante and Stein~\cite{bustastein} proved the conjecture for $r \le 2t + 2$, with Kir\'{a}ly and T\'{o}thm\'{e}r\'{e}sz~\cite{kirtot} extending this to $r \le 4t - 1$. With regards to lower bounds on $\Ryser(r,t)$, the conjecture is trivially tight for $t=r-1$, and Bustamante and Stein~\cite{bustastein} showed that it is also tight for $t=r-2$. However, they demonstrated that it is not always best possible by proving $\Ryser(5, 2) = 2$. More generally, they proved $\Ryser(r,t) \ge \Ryser(\floor{\frac{r}{t}}, 1)$ by observing that replacing every vertex of an $(r',1)$-graph with a set of $t$ vertices gives an $(r't, t)$-graph. Given the aforementioned results on Ryser's conjecture, this shows $\Ryser(r,t) \ge \floor{r/t} - 1$ for many pairs $(r,t)$, and Bustamante and Stein suggested this lower bound should be closer to the truth than the upper bound of Conjecture \ref{conj:weak}.

\subsection{Our results}

Our first result shows that the lower bound of Bustamante and Stein is, in fact, far from optimal. Indeed, we provide a construction, valid for all $t$ and $r$, that greatly improves on the previous lower bound when $t\geq 3$.

\begin{thm} \label{thm:lowerbound}
For all $1 \le t \le r$ we have 
\[\Ryser(r,t)\geq \floor*{\frac{r-t}{2}}+1.\]
\end{thm}

We next prove a matching upper bound when $t$ is large, showing that when $r$ is less than thrice $t$, the true value of $\Ryser(r,t)$ is half the bound of Conjecture~\ref{conj:weak}.

\begin{thm} \label{thm:exactvalue}
For $t,r \in \mathbb{N}$ such that $t+1\leq r \leq  3t-1$, we have 
\[\Ryser(r,t)=\floor*{\frac{r-t}{2}}+1.\]
\end{thm}

Theorem~\ref{thm:lowerbound} gives the lower bound needed for Theorem~\ref{thm:exactvalue}, and hence all that is required is a matching upper bound. In fact, using different arguments, we are able to prove a few upper bounds on $\Ryser(r,t)$. The theorem below collects 
the best upper bounds (excluding the trivial $\Ryser(r,t) \le r-t+1$) that we have in various ranges of the parameters.

\begin{thm} \label{thm:upperbound}
Let $1 \le t \le r$. Then 
\[\Ryser(r,t)\leq \begin{cases}
    \floor*{\frac{r-t}{2}}+1 & \mbox{ if }\quad t\leq r \leq 3t-1, \\
    2r-5t+2 & \mbox{ if }\quad 3t \leq r \le \frac{26t}{7}, \\
    \floor*{\frac{9r-14t}{8}}+2 & \mbox{ if }\quad \frac{26t}{7} \leq r \leq 5t-2, \\
    \floor*{\frac{15r-44t}{8}}+3 & \mbox{ if }\quad 5t-1\leq r \leq \frac{52t-13}{9}. 
\end{cases}\]
\end{thm}

In particular, the first case of the theorem gives the upper bound needed for Theorem~\ref{thm:exactvalue}. To visualise our results, it helps to focus on the asymptotics when $t$ is linear in $r$. To this end,  we define the function $f(\alpha) := \lim\limits_{r \rightarrow \infty} \frac{\Ryser(r,\alpha r)}{r}$. Theorem~\ref{thm:upperbound} can then be seen as a piecewise linear upper bound on $f(\alpha)$. 
Figure~\ref{fig:asymptotics} summarises our knowledge of $f(\alpha)$: we know it exactly for $\alpha \ge \frac13$, while we can still strongly restrict $f(\alpha)$ for smaller values of $\alpha$.

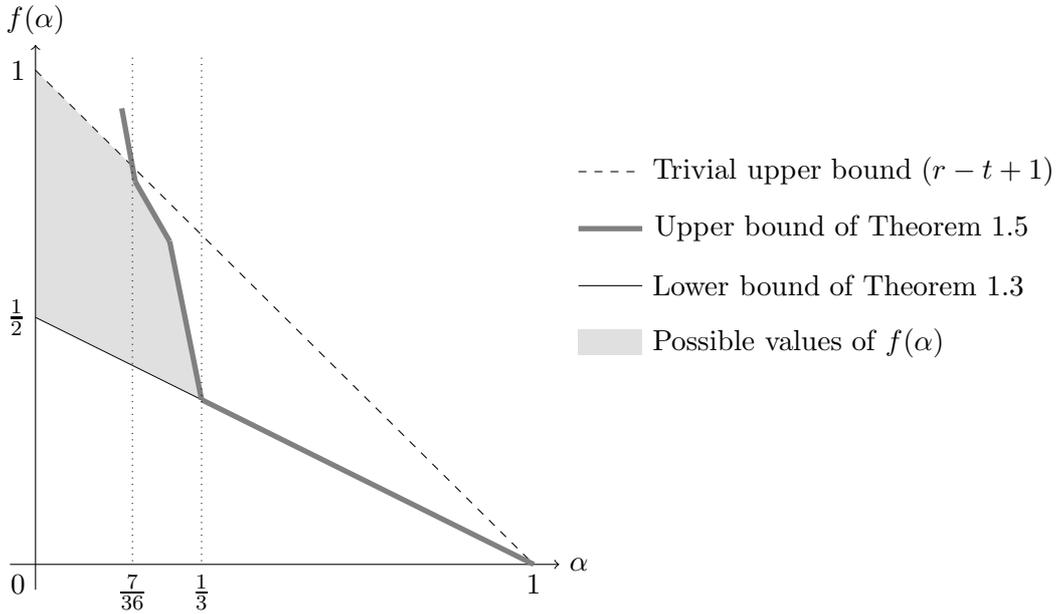
\begin{figure}[h] 
    \centering
    
    \begin{tikzpicture}[scale=0.084]
    
        \draw[->] (-4,0) -- (82,0) node[right] {$\alpha$};
        \draw[->] (0,-4) -- (0,82) node[above] {$f(\alpha)$};
        
        \draw[domain=0:78,dashed,variable=\x] plot ({\x},{78-\x});
        
        \draw[domain=0:78,smooth,variable=\x] plot ({\x},{39-0.5*\x});
        
        \draw[domain=26:78,smooth,line width=2pt,variable=\x,gray] plot ({\x},{39-0.5*\x});
        \draw[domain=21:26,smooth,line width=2pt,variable=\x,gray] plot ({\x},{156-5*\x});
        \draw[domain=15.6:21,smooth,line width=2pt,variable=\x,gray] plot ({\x},{87.75-1.75*\x});
        \draw[domain=13.5:15.6,smooth,line width=2pt,variable=\x,gray] plot ({\x},{146.25-5.5*\x});
        
        \fill[gray, opacity=0.24] (0,39) -- (26,26) -- (21,51) -- (15.6,60.45) -- (15.17,62.83) -- (0,78) -- cycle; 
        
        \draw[dotted] (15.17,80) -- (15.17,0) node[below] {$\frac{7}{36}$};
        
        \draw[dotted] (26,80) -- (26,0) node[below] {$\frac13$};
        
        \node[below] (xint) at (78,0) {$1$};
        \node[below left] (origin) at (0,0) {$0$};
        \node[left] (ylower) at (0,39) {$\frac12$};
        \node[left] (yupper) at (0,78) {$1$};
        
        \draw[dashed] (85,62) -- (95,62) node[right] {Trivial upper bound ($r-t+1$)};
        \draw[line width=2pt, gray] (85,53) -- (95,53) node[right,black] {Upper bound of Theorem~\ref{thm:upperbound}};
        \draw (85,44) -- (95,44) node[right] {Lower bound of Theorem~\ref{thm:lowerbound}};
        \fill[gray,opacity=0.24] (85,37) -- (95,37) -- (95,33) -- (85,33) -- cycle;
        \node[right] (shade) at (95,35) {Possible values of $f(\alpha)$};
    \end{tikzpicture}
    
    \caption{The asymptotics of $\Ryser(r,\alpha r)$.}
    \label{fig:asymptotics}
\end{figure}

In particular, we extend the results of Bustamante and Stein~\cite{bustastein} and Kir\'aly and T\'othm\'er\'esz~\cite{kirtot} by showing Conjecture~\ref{conj:weak} continues to hold for smaller values of $t$, and that it is not tight in most of these cases.

\begin{cor} \label{cor:comparisons}
Conjecture~\ref{conj:weak} holds for all but eight pairs\footnote{The  exceptional pairs, for which Conjecture~\ref{conj:weak} remains open, are $(12,3), (13,3), (16,4), (17,4), (18,4), (22,5), (23,5)$ and $(28,6)$.} $(r,t)$ satisfying $r \le \frac{36t - 17}{7}$.  Furthermore, the conjecture is not tight (that is, $\Ryser(r,t) \le r - t - 1$) for all but $50$ pairs where $t+3 \leq r \le \frac{36t - 25}{7}$.
\end{cor}

The value of $\frac{36t-17}{7}$ comes from comparing when the upper bound on $\Ryser(r,t)$ in the fourth range of Theorem~\ref{thm:upperbound} falls strictly below $r-t+1$, using the fact that the cover number must be an integer. This comparison describes exactly when our bounds become trivial for all large values of $r$ and $t$. However, for  small values of $r$ and $t$, we may fall in other ranges and so have to check a finite number of cases for exceptions. In doing so, we note that for certain values of $(r,t)$ we must appeal to the previously mentioned result of Kir\'aly and T\'othm\'er\'esz~\cite{kirtot} who proved the conjecture  whenever $r\leq 4t-1$.  The range and exceptions for when the conjecture is not tight are also calculated in a similar manner, using Theorem~\ref{thm:upperbound}. 

These results beg the question of what the true value of $\Ryser(r,t)$ should be; we discuss this further in Section~\ref{sec:variants}, and propose a new conjecture in Conjecture~\ref{conj:strong}.

\subsection{$k$-wise intersecting hypergraphs}

In the above results, we require that all pairwise intersections of the edges of the hypergraphs have size at least $t$. A natural stronger condition is to impose the same restriction on all $k$-wise intersections of edges, rather than just pairwise. This setting has often been studied in the extremal combinatorics literature. Frankl~\cite{frankl76} first studied such hypergraphs, determining the maximum number of edges possible when all $k$-wise intersections are non-empty. S\'os~\cite{sos73} then raised the problem of finding the largest hypergraphs where the sizes of all $k$-wise intersections lie in some set $L$, and various results in this direction were obtained by F\"uredi~\cite{furedi83}, Vu~\cite{vu97, vu99}, Grolmusz~\cite{grolmusz02}, Grolmusz and Sudakov~\cite{grolsuda02}, F\"uredi and Sudakov~\cite{furesuda04} and Szab\'o and Vu~\cite{szabovu05}.

We say a hypergraph $\m H$ is $k$-wise $t$-intersecting if, for any edges $e_1, e_2, \hdots, e_k \in E(\m H)$, we have $\card{\cap_{i=1}^k e_i} \ge t$. Following Ryser's Conjecture, we study how much smaller a cover one is guaranteed to find in $r$-uniform $r$-partite hypergraphs satisfying the more restrictive condition of being $k$-wise $t$-intersecting. In a stroke of serendipity, the range of intersection sizes for which Theorem~\ref{thm:exactvalue} holds is precisely what is needed to give an exact answer in this setting.

\begin{thm} \label{thm:kwise}
Let $\m H$ be an $r$-uniform $r$-partite $k$-wise $t$-intersecting hypergraph.  If $k \ge 3$ and $t \ge 1$, or $k = 2$ and $t > \tfrac{r}{3}$, then
\[ \tau(\m H) \le \left\lfloor \frac{r-t}{k} \right\rfloor + 1, \]
and this bound is best possible.
\end{thm}

\paragraph{Organisation of the paper.}
We prove the above theorems in the following section: the lower bound of Theorem~\ref{thm:lowerbound} is derived in Section~\ref{sec:lowerboundproof}, the upper bounds of Theorem~\ref{thm:upperbound} are proven in Section~\ref{sec:upperboundproof}, and Theorem~\ref{thm:kwise} is deduced in Section~\ref{sec:kwise}. Thereafter we suggest several directions for further research in Section~\ref{sec:variants} by presenting initial results on variants of the problem where we require the hypergraphs to be strictly $t$-intersecting or we try to cover each edge of an $(r,t)$-graph at least $s$ times.

\section{Proofs of the main results}

In this section we prove our main results, Theorems~\ref{thm:lowerbound},~\ref{thm:upperbound} and~\ref{thm:kwise}, by establishing lower (Section~\ref{sec:lowerboundproof}) and upper (Section~\ref{sec:upperboundproof}) bounds on the extremal function $\Ryser(r,t)$, and then extending them to $k$-wise $t$-intersecting hypergraphs (Section~\ref{sec:kwise}).

\subsection{Lower bound construction} \label{sec:lowerboundproof}

To obtain the lower bound, and thereby prove Theorem~\ref{thm:lowerbound}, we need to construct $(r,t)$-graphs with large cover numbers. The hypergraphs we consider are of the following form.

\begin{dfn} \label{def:1levelconstruction}
For $0\leq\ell \leq r-1$, we define $\cH^r_\ell$ to be the following $r$-uniform $r$-partite  hypergraph. Let $m=\binom{r}{r-\ell}$ and fix some ordering $\binom{[r]}{r-\ell}=\{S_1,\ldots,S_m\}$ of the $(r-\ell)$-subsets of $[r]$. We define  
\[V(\cH^r_\ell):= \{0,1,\ldots,m\} \times [r] \quad \mbox{ and } \quad E(\cH^r_\ell):=\{e_i: i \in[m]\},\]
where $e_i=\{(0,j):j\in S_i\}\cup \{(i,j):j\in [r]\setminus S_i\}$, for each $i\in[m]$.
\end{dfn}

Note that $\cH^r_\ell$ is indeed  $r$-partite with  parts $P_j=\{(i,j):0\leq i \leq m \}$ for $1\leq j\leq r$. We now show that choosing $\ell$ appropriately gives a construction verifying Theorem \ref{thm:lowerbound}. 

\begin{prop} \label{prop:lowerboundcalcs} 
For $0 \le \ell \le \floor{\frac{r-1}{2}}$, $\m H^r_{\ell}$ is $(r - 2 \ell)$-intersecting with $\tau(\m H^r_{\ell}) = \ell + 1$.
\end{prop}
\begin{proof}
To see that $\m H^r_{\ell}$ is $(r-2\ell)$-intersecting, observe that each edge of $\m H^r_{\ell}$ misses exactly $\ell$ vertices from the set $L_0 = \{(0,j) : j \in [r]\}$. It then follows that, for any two edges $e_i, e_{i'} \in E(\m H^r_{\ell})$, we have
\[ \card{e_i\cap e_{i'}} \geq \card{e_i\cap e_{i'}\cap L_0} \geq \card{L_0} - 2\ell = r-2\ell. \]

We now establish the cover number of $\cH_\ell^r$. Assume for a contradiction that $\cH^r_\ell$ has a cover $C\subset V(\cH^r_\ell)$ of size $c\leq \ell$. First we show that we may assume that $C\subset L_0$. Indeed, if $v = (i,j) \in C$ for some $i \ge 1$, then the only edge that could contain $v$ is $e_i$.  If we replace $v$ with any  vertex in $e_i\cap L_0$, the modified set $C$ has not increased in size and still covers $\m H^r_\ell$.

Since $|C|=c\leq \ell$, we have $|L_0\setminus C|\geq r-\ell$ and hence there exists an $i^*\in[m]$ such that $\{0\}\times S_{i^*}\subseteq L_0\setminus C$. Then $e_{i^*}\cap C=\emptyset$, contradicting the fact that $C$ is a cover. Thus $\tau(\cH^r_\ell)\geq \ell+1$. To see we have equality, note that any subset of $\ell+1$ vertices in $L_0$ forms a cover. 
\end{proof}

The key property needed in the above proof is that, for each subset $S \subseteq L_0$ of size $r-\ell$, there is an edge of $\m H^r_\ell$ intersecting $L_0$ exactly at the vertices of $S$. Our construction is edge-minimal with respect to this key property, and further ensures that all vertices not in $L_0$ have degree at most one, leading to an easy proof of the cover number. 

However, as long as the key property is maintained, there is great flexibility in how the rest of the hypergraph is constructed.  For instance, one can instead make it vertex-minimal, having parts of size $\ell + 1$ rather than $\binom{r}{r-\ell} + 1$, so that each part $P_j$ also forms a minimum vertex cover. We omit the details of this construction for the sake of brevity, as we already have all we need to prove Theorem~\ref{thm:lowerbound}.

\begin{proof}[Proof of Theorem~\ref{thm:lowerbound}]
Set $\ell = \floor{\frac{r-t}{2}}$. By Proposition~\ref{prop:lowerboundcalcs}, $\m H^r_{\ell}$ is $(r - 2\ell)$-intersecting, and as $r - 2\ell \ge t$, it follows that $\m H^r_{\ell}$ is an $(r,t)$-graph.  The proposition further asserts that $\tau(\m H^r_{\ell}) = \ell + 1$, and thus
\[ \Ryser(r,t) \ge \tau(\m H^r_{\ell}) = \floor*{\frac{r-t}{2}} + 1. \qedhere \]
\end{proof}

\subsection{Upper bounds} \label{sec:upperboundproof}
In this section we prove Theorem~\ref{thm:upperbound}. 
These upper bounds are derived from a sequence of results obtained by considering configurations of two or three edges of the hypergraph. Before proceeding, we fix some notation that will be useful in what follows. 

\begin{dfn} \label{def:degreesums}
Let $\cH$ be an $r$-uniform $r$-partite hypergraph with parts $P_j$ for $1\leq j\leq r$. Suppose that $e_1,\ldots,e_k\in E(\cH)$. Then for   $v\in V(\cH)$, we define
\[d(v;e_1,\ldots,e_k)=\card{\{i\in[k]:v\in e_i\}}\]
to be the degree of $v$ with respect to the $k$ edges $e_1,\ldots , e_k$. Also, given a vertex subset $C\subseteq V(\cH)$ not wholly containing any part $P_j$, we define 
\[\Delta_\cH(C;e_1,\ldots,e_k) = \sum_{j=1}^r \max_{v\in P_j\setminus C}d(v;e_1,\ldots,e_k)\]
to be the maximum sum of degrees (with respect to $e_1,\ldots,e_k$) when we take one vertex from each part and avoid $C$.
\end{dfn}

The utility of this definition comes from the following easy observation, which we use repeatedly in the subsequent proofs.
\begin{obs} \label{obs:cover} Suppose that $\cH$ is an $r$-uniform $r$-partite hypergraph, $C\subset V(\cH)$ and $e_1,\ldots,e_k\in E(\cH)$ are edges of $\cH$. Now if $f\in E(\cH)$ and $f\cap C=\emptyset$, we have 
\begin{equation} \label{eq:keyobs}
    \sum_{i=1}^k|f\cap e_i|\leq \Delta_\cH(C;e_1,\ldots,e_k).
\end{equation}
Consequently, if $\cH$ is $t$-intersecting 
 and $\Delta_\cH(C;e_1,\ldots,e_k)\leq kt-1$, then $C$ is a cover for $\cH$. Indeed, there can be no $f\in E(\cH)$ disjoint from $C$, as by~\eqref{eq:keyobs} and the pigeonhole principle, there would be some $i\in[k]$ for which $\card{e_i\cap f} \leq t-1$, contradicting $\cH$ being $t$-intersecting.
\end{obs}

Armed with this observation, we can prove upper bounds on the cover numbers of $(r,t)$-graphs. In the following lemma, we begin by constructing a cover consisting of vertices lying in two edges of such a hypergraph. 

\begin{lem} \label{lem:2edges}
Let $\cH$ be an $(r,t)$-graph, let $e_1,e_2\in E(\cH)$, and set $t' = |e_1\cap e_2| \geq t$. Then 
\[\tau(\cH)\leq \begin{cases}
    \floor*{\frac{r-t'}{2}} + t'-t+1 & \mbox{ if } r-2t+1 \leq t'\leq r, \\
    2r-4t-t'+2 & \mbox{ if } t\leq t'\leq r-2t.
\end{cases} \]
\end{lem}
\begin{proof}
Let the parts of $\cH$ be $P_j$ for $j\in[r]$. If $t'\geq r-2t+1$ then $s=\floor{\frac{r-t'}{2}} + t'-t+1$ satisfies $1 \le s \le t'$. We claim that an arbitrary set $C_1$ of $s$ vertices from $e_1\cap e_2$  is a cover. Indeed,
\[\Delta_\cH(C_1;e_1,e_2)= \sum_{\substack{j\in[r]:\\e_1\cap P_j=e_2\cap P_j, \\ 
e_1\cap P_j \notin C_1}}2+ \sum_{\substack{j\in[r]:\\e_1\cap P_j\neq e_2\cap P_j}}1 =  2\card{(e_1\cap e_2)\setminus C_1}+r-t' = r + t' - 2s \leq 2t-1,\]
and thus the conclusion follows from Observation~\ref{obs:cover}.

Now consider the case where $t'\leq r-2t$.  Without loss of generality, we may assume the intersection of $e_1$ and $e_2$ is contained in the first $t'$ parts, labelling the vertices of $e_1$ as $\{u_1, \dots, u_r\}$ and of $e_2$ as $\{u_1, \dots, u_{t'}, v_{t'+1}, \dots, v_r\}$, where $u_j, v_j \in P_j$ for all $j$. 
Letting
\[ C_2 = \{u_1, \hdots, u_{t'}\} \cup \bigcup_{j = t' + 1}^{r - 2t + 1} \{u_j, v_j \},\]
we have $\Delta_{\m H}(C_2; e_1, e_2) = 2t - 1$. By Observation~\ref{obs:cover}, we can again deduce that $\tau(\m H) \le \card{C_2} = 2(r - 2t + 1) - t' = 2r - 4t - t' + 2$.
\end{proof}

Lemma~\ref{lem:2edges} will suffice to prove the first two parts of Theorem~\ref{thm:upperbound}.  For the latter parts, we shall need to consider covers consisting of vertices lying in three edges instead. Before proceeding, though, we present a reformulation of Lemma~\ref{lem:2edges} that will be more convenient for later proofs.

\begin{cor} \label{cor:2edgeagain}
Let $\eta\in \mathbb{N}$ and  let $\cH$ be an $(r,t)$-graph such that $\tau(\cH)\geq\eta+1$. Then, for all $e,f\in E(\cH)$, we have that either \begin{itemize}
    \item[$(i)$] $\card{e\cap f} \geq  2\eta +2t - r$, or
    \item[$(ii)$] $\card{e\cap f} \leq 2r-4t-\eta+1$.
\end{itemize}
\end{cor}
\begin{proof}
Let $e,f\in E(\cH)$ be an arbitrary pair of edges of $\cH$ and take $t'=|e\cap f|$. If $t'\geq r-2t+1$, then we claim that in fact $t'\geq 2 \eta +2t-r$. Indeed, if $t'\leq2\eta +2t-r-1$, then Lemma~\ref{lem:2edges} implies that \[\tau(\cH)\leq \floor*{\frac{r-t'}{2}}+t'-t+1\leq \frac{r + t'}{2}-t+1\leq\eta + \frac{1}{2}, \]
a contradiction. 

Hence, if $(i)$ is not satisfied for $t'=\card{e\cap f}$, we must have $t'\leq r-2t$.  By Lemma~\ref{lem:2edges}, it follows that $\eta + 1 \le \tau(\m H) \le 2r - 4t - t' + 2$, from which we deduce that $t' \le 2r - 4t - \eta + 1$.
\end{proof}

The following lemma is the analogue of Lemma~\ref{lem:2edges} when constructing covers from vertices that lie in three fixed edges, as opposed to only using two edges. 

\begin{lem} \label{lem:3edges}
Let $r \ge 3t$, let $\m H$ be an $(r,t)$-graph, and let $e_1,e_2,e_3 \in E(\m H)$.  Set $t_1 = \card{e_1\cap e_2\cap e_3}$ and $t_2 = \card{e_1\cap e_2 \setminus e_3}+\card{e_1\cap e_3 \setminus e_2}+ \card{e_2\cap e_3 \setminus e_1}$. Then
\[ \tau(\m H)\leq \begin{cases}
    \frac13 (2t_1+t_2+r-3t+3) & \text{ if } r-3t+1+t_2 \leq t_1\leq r, \\
    r-3t+1+t_2 & \text{ if }  r-3t+1-t_2\leq t_1\leq r-3t+1+t_2, \\
    3r-2t_1-t_2-9t+3 & \text{ if }  0\leq t_1\leq r-3t+1-t_2.
\end{cases} \]
\end{lem}

\begin{proof}
Observe that $t_1$ counts the number of vertices that are in all three edges, while $t_2$ counts the number of vertices in precisely two of the three edges.  Let $T = e_1 \cap e_2 \cap e_3$ be the set of $t_1$ vertices contained in all three edges and let $D$ be the set of $t_2$ vertices contained in exactly two of the three edges.   
 When building a transversal of the parts that intersects $e_1, e_2$ and $e_3$ as much as possible, it is optimal to choose as many vertices from $T$ as possible, followed by vertices from $D$.  Therefore, to minimise $\Delta_{\m H}(C; e_1, e_2, e_3)$, we will choose $C$ so as to first block the vertices in $T$, followed by those in~$D$.

Let us first consider the case when $t_1\geq r-3t+1+t_2$. Setting $s_1 = t_1 - \floor*{\frac13 (t_1 - t_2 - r + 3t - 1)} \le \frac13 (2t_1 + t_2 + r - 3t + 3)$, note that $1 \le s_1 \le t_1$.  Taking $C_1$ to be an arbitrary subset of $T$ of size $s_1$, we have
\[ \Delta_{\m H}(C_1;e_1,e_2,e_3)= 3\card{(e_1\cap e_2\cap e_3)\setminus C_1} + 2t_2 +(r-t_1-t_2)\leq 3t-1, \]
since we cannot select a vertex of $e_1 \cup e_2 \cup e_3$ in the $s_1$ parts spanned by $C_1$. It thus follows from Observation~\ref{obs:cover} that $C_1$ is a cover for $\m H$, giving the claimed bound on $\tau(\m H)$.

Next, suppose $r-3t+1\leq t_1\leq r-3t+1+t_2$, and set $s_2=r-3t+1+t_2-t_1$, noting that $0\leq s_2\leq t_2$. Take $C_2=T \cup S_2$, where $S_2$ is a subset of $D$ of size $s_2$. 
Consider a transversal of the parts that is disjoint from $C_2$. There are $t_2-s_2$ parts (those intersecting $D\setminus S_2$) in which the transversal could intersect up to two of the three edges $e_1,e_2$ and $e_3$. In all other parts  the transversal can intersect at most one of the three edges and there are $t_1$ parts (those that intersect $T$) in which the transversal must be disjoint from all three edges.  
Thus
\[\Delta_\cH(C_2;e_1,e_2,e_3)=2(t_2-s_2)+(r-t_1-(t_2-s_2))=3t-1,\]
and so, by Observation~\ref{obs:cover}, $C_2$ covers $\m H$, showing $\tau(\m H) \le \card{C_2} = r - 3t + 1 + t_2$.

In the range $r-3t+1-t_2\leq t_1\leq r-3t$, set $s_3=r-3t+1-t_1$, whence $1\leq s_3\leq t_2$. We define $D'$ to be the $t_2$ vertices which are contained in exactly one of the edges $e_1, e_2$ and $e_3$ and lie in parts which intersect $D$.  Now take $C_3= T \cup D \cup S_3$ where $S_3$ is an arbitrary subset of $D'$ of size $s_3$.  A transversal disjoint from $C_3$ can then only meet the edges $e_1, e_2$ and $e_3$ in the $t_2 - s_3$ vertices of $D' \setminus S_3$, as well as in the $r - t_1 - t_2$ parts where the three edges are pairwise disjoint. We therefore have
\[\Delta_{\m H}(C_3;e_1,e_2,e_3)=(t_2-s_3)+(r-t_1-t_2)=3t-1, \]
and so Observation~\ref{obs:cover} implies $C_3$ is a cover of the stated size.

Finally, we are left with the case when $t_1\leq r-3t+1-t_2$, for which we set $s_4 = r - 3t + 1 - t_1 - t_2$. Take $C_4 = T \cup D \cup D' \cup S_4$, where $D$ and $D'$ are defined as above and $S_4$ consists of the $3s_4$ vertices of $e_1 \cup e_2 \cup e_3$ from $s_4$ of the parts where the edges $e_1, e_2$ and $e_3$ are pairwise disjoint.  Then any transversal disjoint from $C_4$ can only meet the three edges in the $r - t_1 - t_2 - s_4$ parts not spanned by $C_4$, from which it follows that $\Delta_{\m H}(C_4; e_1, e_2, e_3) = 3t - 1$. By Observation~\ref{obs:cover}, $C_4$ is a cover of $\m H$, and so $\tau(\m H) \le \card{C_4} = t_1 + 2t_2 + 3s_4 = 3r - 2t_1 - t_2 - 9t + 3$.
\end{proof}

By applying Lemma~\ref{lem:3edges} in conjunction with Corollary~\ref{cor:2edgeagain}, we will prove the following upper bounds on the cover numbers of $(r,t)$-graphs.

\begin{prop} \label{prop:3edgeupperbound}
For all $t \ge 1$ and $r \ge 3t$, we have 
\[  \Ryser(r,t)\leq \begin{cases}
    \ceil*{\frac{5r-10t+2}{4}} +\ceil*{\frac{6t-r-1}{8}} & \mbox{ if } 3t\leq r \leq 5t-2, \\
    \ceil*{\frac{3r-1}{4}} + \ceil*{\frac{9r-44t+13}{8}} & \mbox{ if } 5t-1 \leq r \leq \frac{52t-13}{9}.
\end{cases} \]
\end{prop}

Before proving this proposition, we observe that we have the necessary bounds to establish our main result.

\begin{proof}[Proof of Theorem~\ref{thm:upperbound}]
We prove the theorem by induction on $r - t$.  For the base case, we have $t = r$.  Trivially, an $(r,r)$-graph can have at most one edge, and thus can be covered by a single vertex.  Thus $\Ryser(r,r) = 1$, as stated in the first case of the theorem.

For the induction step, suppose $r - t \ge 1$, and let $\m H$ be an $(r,t)$-graph of maximum cover number, so that $\Ryser(r,t) = \tau(\m H)$. As previously stated, we shall use Lemma~\ref{lem:2edges} to prove the first two cases of the theorem. To this end, let $e_1, e_2 \in E(\m H)$ be a pair of edges with the smallest intersection. We may assume that $\card{e_1 \cap e_2} = t$, as otherwise $\m H$ is in fact an $(r,t+1)$-graph, and we are done by induction (as our upper bound on $\Ryser(r,t)$ is decreasing in $t$).

Applying Lemma~\ref{lem:2edges} with the edges $e_1$ and $e_2$, we have $t' = t$, from which it follows that
\[ \Ryser(r,t) = \tau(\m H) \le \begin{cases}
    \floor*{\frac{r-t}{2}} + 1 & \mbox{ if } r - 2t + 1 \le t \le r, \\
    2r - 5t + 2 & \mbox{ if } t \le r - 2t.
\end{cases} \]
Simplifying the ranges for which these bounds hold, we see that the first is valid when $r \le 3t - 1$, as required for the first case of the theorem, while the second bound above is valid provided $r \ge 3t$.

The latter two cases of the theorem are direct consequences of Proposition~\ref{prop:3edgeupperbound}, which we can apply whenever $r \ge 3t$. To simplify the bounds, we estimate the ceiling terms, obtaining
\[ \ceil*{\frac{5r-10t+2}{4}} + \ceil*{\frac{6t-r-1}{8}} \le \frac{5r - 10t + 5}{4} + \frac{6t - r + 6}{8} = \frac{9r - 14t}{8} + 2 \]
and
\[ \ceil*{\frac{3r-1}{4}} + \ceil*{\frac{9r-44t+13}{8}} \le \frac{3r + 2}{4} + \frac{9r - 44t + 20}{8} = \frac{15r - 44t}{8} + 3. \]
That is, Proposition~\ref{prop:3edgeupperbound} implies that whenever $r \ge 3t$, we have
\[ \Ryser(r,t) \le \begin{cases}
    \frac{9r - 14t}{8} + 2 & \mbox{ if } 3t \le r \le 5t - 2, \\
    \frac{15r - 44t}{8} + 3 & \mbox{ if } 5t - 1 \le r \le \frac{52t - 13}{9}.
\end{cases} \]
This matches the latter two cases of Theorem~\ref{thm:upperbound}.  Finally, we note that the bound $\frac{9r - 14t}{8} + 2$ improves the earlier bound of $2r - 5t + 2$ whenever $r \ge \frac{26}{7}t$, justifying the endpoints of the ranges of the second and third cases in the theorem.
\end{proof}

All that remains is the proof of Proposition~\ref{prop:3edgeupperbound}, which we delay no further.

\begin{proof}[Proof of Proposition~\ref{prop:3edgeupperbound}]
In the first case, let $r$ and $t$ be such that $3t\leq r \leq 5t-2$, and define
\[x =\ceil*{\frac{5r-10t + 2}{4}} \quad \mbox{ and } \quad z =\ceil*{\frac{6t-r-1}{8}}.\]
Note that  \begin{equation} \label{eq:zandt}
    1\leq z \leq t \qquad \mbox{ and } \qquad 0\leq x \leq r-t,
\end{equation}
using here that $3t\leq r \leq 5t-2$.
Suppose for contradiction that $\m H$ is an $(r,t)$-graph with $\tau(\cH)\geq x+z+1$. Applying Corollary \ref{cor:2edgeagain}, any pair of edges $e,f\in E(\cH)$ must satisfy
\begin{equation} \label{eq:2edgerestriction}
    \card{e\cap f} \geq  2x+2z +2t -r \qquad
\mbox{ or } \qquad 
\card{e\cap f} \leq 2r-4t-x-z+1.
\end{equation}
Now take $e_1, e_2\in E(\cH)$ to be two edges such that $\card{e_1\cap e_2}=t$ (again, we may assume such a pair exists, as otherwise $\m H$ is an $(r,t+1)$-graph and we obtain a stronger bound on $\tau(\m H)$).  Let $Z\subseteq e_1\cap e_2$ be a set of $z$ vertices and $X \subseteq e_1 \setminus e_2$ a set of $x$ vertices, noting that this is possible due to~\eqref{eq:zandt}. We take $Y=X\cup Z$ and note that $Y$ intersects both $e_1$ and $e_2$ (as $z\geq 1$) and is not a cover as it has size exactly $x+z$ and we assumed that $\tau(\cH)\geq x+z+1$. Therefore, there exists an edge $e_3 \in E(\m H)$ such that $e_3\cap Y=\emptyset$. We define $a,b,c \in \mathbb{N}$ as follows:
\[ a = \card{e_1\cap e_2 \setminus e_3}, \quad b = \card{e_1\cap e_3 \setminus e_2} \quad \mbox{and} \quad c = \card{e_2\cap e_3\setminus e_1}. \]
Observe that $\card{e_1\cap e_2 \cap e_3}=t-a$. Using the fact that $e_3 \cap Y=\emptyset$, we can derive some bounds on the parameters $a, b$ and $c$. Indeed, since $e_3$ is disjoint from $Z$, $a \ge z$.  As $\card{e_1 \cap e_3} = b + (t-a) \ge t$, we must further have $b \ge a$, while considering $\card{e_2 \cap e_3}$ similarly shows $c \ge a$.  Finally, as $e_3$ is disjoint from $X$, we must have $b \le r - t - x$.  Putting this all together, we have
\begin{equation} \label{eq:spaceboundaries}
    z \leq a \leq b \leq r-t-x \quad \mbox{and} \quad a \le c. 
\end{equation}

A further restriction on the parameters comes from considering $\card{e_2\cap e_3}=(t-a)+c$. We have from~\eqref{eq:spaceboundaries} that $b \geq z$ and, since $\cH$ is $r$-partite, in each of the $b$ parts which contain vertices of $e_1 \cap e_3 \setminus e_2$, there are no vertices which lie in $e_2 \cap e_3$. Moreover $e_3$ and $e_2$ are also disjoint in the $z$ parts which host vertices of $Z$. Thus we can conclude that 
\[t+c-a= \card{e_2\cap e_3} \leq r-b-z \leq r-2z. \]
Due to the fact that
\begin{equation} \label{eq:nobigintersection}
2x+2z+2t-r\geq \frac{5r-6t+3}{4} > \frac{5r-6t+1}{4}\geq r-2z,
\end{equation}
it follows from~\eqref{eq:2edgerestriction} that we must in fact have
\begin{equation} \label{eq:smallintersection}
    t+c-a= \card{e_2\cap e_3} \leq 2r-4t-x-z+1.
\end{equation}
We now look to apply Lemma~\ref{lem:3edges} to the three edges $e_1,e_2$ and $e_3$ to show that $\tau(\cH) \leq x +z$, thus reaching a contradiction. To this end, note that in the notation of Lemma~\ref{lem:3edges}, we have $t_1=t-a$ and $t_2=a+b+c$. First suppose that our parameters fall into the first range given by the upper bound in Lemma~\ref{lem:3edges}. That is, $t-a \geq r-3t+1+a+b+c$ or, rearranging, 
\begin{equation} \label{eq:case1weakbound}
    2a+b+c \leq 4t-r-1.
\end{equation}
We then have that 
\begin{equation}
    \label{eq:case1bound}\tau(\cH) \leq \frac{r+b+c-a-t+3}{3} \leq \frac{3t-3a+2}{3}\leq \frac{3t+2}{3}\leq x+z,  \end{equation}
using~\eqref{eq:case1weakbound} in the second inequality, the fact that $a\geq 0$ in the third and the fact that $r\geq 3t$ in the last inequality.  

Now we turn to the second case of Lemma~\ref{lem:3edges} and observe that the given bound  is \begin{align*}
     \tau(\cH)&\leq r-3t+1+a+b+c \\ &\leq r-3t+1 +a +b + (2r-5t-x-z+1+a) \\ &\leq 3r-8t+2-x-z+3b  \\ &\leq 3r-8t+2-x-z+3(r-t-x)\\ &=6r-11t-4x-z+2, 
     \end{align*}
where we used~\eqref{eq:smallintersection} to bound $c$ in the second inequality, and the bounds $a \le b$ and $b \le r - t - x$ from~\eqref{eq:spaceboundaries} in the third and fourth inequalities respectively. One has that 
\begin{equation} \label{eq:case2bound}
    5x+2z\geq 5\left(\frac{5r-10t + 2}{4}\right)+2 \left(\frac{6t-r-1}{8}\right)\geq 6r-11t+2,
\end{equation}
and hence $\tau(\cH)\leq x+z$ in this case too.

Finally, in the third case of Lemma~\ref{lem:3edges}, we have
\[ \tau(\cH)\leq 3r-11t+3+a-b-c\leq 3r-11t+3-a \leq 3r-11t+3-z, \]
 using~\eqref{eq:spaceboundaries} to bound $b, c \ge a$ in the second inequality and $a \ge z$ in the last inequality. As 
\begin{equation} \label{eq:case3bound}
    x+2z\geq \frac{5r-10t+2}{4}+\frac{6t-r-1}{4} = r - t + \frac14 \geq 3r-11t+3 
\end{equation}
for all $r\leq 5t-2$, we can conclude that $\tau(\cH)\leq x+z$ in this case as well. Therefore, we have shown $\tau(\cH)\leq x+z$, providing the contradiction needed to complete the proof of the first bound. 

\medskip

The proof of the second bound is almost identical to that of the first and so we omit the details. The difference here comes as we define $x$ and $z$ as follows:
\[x =\ceil*{\frac{3r-1}{4}} \quad \mbox{ and } \quad z = \ceil*{\frac{9r-44t+13}{8}}.\]
The rest of the proof goes through verbatim and one simply has to check that the inequalities~\eqref{eq:zandt},~\eqref{eq:nobigintersection}, \eqref{eq:case1bound},~\eqref{eq:case2bound} and~\eqref{eq:case3bound} all hold.  One has to use the fact that $5t-1 \leq r \leq \frac{52t-13}{9}$ in order to prove~\eqref{eq:zandt},~\eqref{eq:nobigintersection} and~\eqref{eq:case1bound}. The lower bound on $r$ is necessary for~\eqref{eq:nobigintersection} to hold, whilst the upper bound on $r$ is necessary so that the upper bound on $z$ in~\eqref{eq:zandt} always holds. Given this, we can again conclude $\tau(\m H) \le x + z$.
\end{proof}

\subsection{$k$-wise intersecting hypergraphs} \label{sec:kwise}

As we shall show now, our exact results for $\Ryser(r,t)$ allow us to obtain tight bounds on the cover numbers of $k$-wise $t$-intersecting $r$-uniform $r$-partite hypergraphs.

\begin{proof}[Proof of Theorem~\ref{thm:kwise}]
We prove the upper bound by induction on $k$.  The base case, when $k = 2$ (and $t > \tfrac{r}{3}$), is the first case of Theorem~\ref{thm:upperbound}.

For the induction step, we have $k \ge 3$ and $t \ge 1$.  Given $k-1$ edges $e_1, \hdots, e_{k-1} \in E(\m H)$, let $U = \cap_{i=1}^{k-1} e_i$.  The $k$-wise intersection condition implies that every edge meets $U$ in at least $t$ elements.  Thus, if $B$ is obtained by removing $t-1$ elements from $U$, $B$ must be a cover for $\m H$.

If there are $k-1$ edges whose intersection has size at most $\left\lfloor \frac{r-t}{k} \right \rfloor + t$, we are done.  We may therefore assume $\m H$ is $(k-1)$-wise $t'$-intersecting, where $t' = \left\lfloor \frac{r-t}{k} \right\rfloor + t + 1$.  Note that if $k = 3$, then
\[ t' = \left\lfloor \frac{r-t}{3} \right\rfloor + t + 1 \ge \frac{r+2t}{3} > \frac{r}{3}. \]

Hence, by induction, $\tau(\m H) \le \left \lfloor \frac{r-t'}{k-1} \right \rfloor + 1$.  Define integers $a,b$ such that $0 \le b \le k-1$ and  $r - t = ak + b$ and note that it follows from the definition of $t'$ that  $t' = a + t + 1$.  We then have
\[ \left \lfloor \frac{r-t'}{k-1} \right\rfloor + 1 = \left \lfloor \frac{ r - a - t - 1 }{k-1} \right \rfloor + 1 = \left\lfloor \frac{a(k-1) + b - 1}{k-1} \right\rfloor + 1 \le a + 1 = \left\lfloor \frac{r-t}{k} \right\rfloor + 1, \]
completing the induction.

To finish, we show that the bound is best possible.  Setting $\ell = \left\lfloor \frac{r-t}{k} \right\rfloor$, consider the hypergraph $\m H_{\ell}^r$ from Definition~\ref{def:1levelconstruction}.  To see that $\m H_{\ell}^r$ is $k$-wise $t$-intersecting, observe that each edge misses $\ell$ vertices of the form $(0,j)$.  Hence, in the intersection of $k$ edges, we can miss at most $k \ell$ of these $r$ vertices, and thus the $k$ edges must intersect in at least $r - k \ell \ge t$ vertices, as required.  By Proposition~\ref{prop:lowerboundcalcs}, $\tau(\m H_{\ell}^r) = \ell + 1 = \left\lfloor \frac{r-t}{k} \right\rfloor + 1$, matching the upper bound.
\end{proof}

\section{Further variants and open problems} \label{sec:variants}

In this paper, we have studied Ryser's Conjecture for $t$-intersecting hypergraphs.  In particular, we have shown $\Ryser(r,t) = \floor{\frac{r-t}{2}} + 1$ whenever $r \le 3t - 1$, and have proved Conjecture~\ref{conj:weak} for all but finitely many pairs $(r,t)$ satisfying $r \le \frac{36t - 17}{7}$.  Given these results, it is natural to ask what happens when $r$ is larger with respect to $t$.

Since the upper bounds of Theorem~\ref{thm:upperbound} are obtained by considering configurations of two and three edges (Lemmas~\ref{lem:2edges} and~\ref{lem:3edges} respectively), the obvious next step is to prove an analogous result for configurations of four edges. However, as one increases the number of edges in the configuration, the number of variables (representing the intersections of these edges) grows exponentially and one has much less control over the values that these sizes of intersections can have. Indeed, even with just four edges, we could not see a way to channel our ideas to get a stronger upper bound.

Another approach to understanding the behaviour of $\Ryser(r,t)$ is to try and determine the value of the function for small values of $r$ and $t$, using this as a testing ground for new ideas to give more general proofs. We considered the smallest open cases: $3 \le \Ryser(6,2) \le 4$ and $ 3\le \Ryser(7,2) \le 5$, where the lower bounds follow from Theorem~\ref{thm:lowerbound}, the upper bound on $\Ryser(6,2)$ follows from Theorem~\ref{thm:upperbound}  and the upper bound on $\Ryser(7,2)$ follows from the work of  Kir\'aly and T\'othm\'er\'esz~\cite{kirtot}. We managed to improve these upper bounds, showing that $\Ryser(6,2)=3$ and $\Ryser(7,2)\le 4$. Unfortunately, the arguments for these two new bounds required ad hoc methods and, as we doubt such arguments will lead to a significantly wider range of results, we have chosen to omit these proofs.

With regards to the broader picture, we concede that it may be challenging to resolve Conjecture~\ref{conj:weak} in full, since the case $t = 1$ is the intersecting case of Ryser's Conjecture itself.  As discretion is the better part of valour, one might restrict one's attention to $t \ge 2$ and seek in these cases to prove the conjecture and, with a bit more ambition, to fully determine $\Ryser(r,t)$.  In this range, given the lack of a better construction, Theorem~\ref{thm:exactvalue}, and the $k$-wise result of Theorem~\ref{thm:kwise}, we propose the following conjecture.

\begin{conj} \label{conj:strong}
For all $2 \le t \le r$,
\[ \Ryser(r,t) = \floor*{\frac{r-t}{2}} + 1. \]
\end{conj}

If we are to be honest, it is only a proper (but non-empty) subset of the authors that fully believes in this conjecture.  That said, we are all happy to pose it, in the hopes of provoking the community into finding a proof or a counterexample.  Should the conjecture be true, it would represent a marked difference between the intersecting and $t$-intersecting ($t \ge 2$) versions of Ryser's Conjecture.  Though this may be surprising at first sight, such discrepancies are not unheard of in extremal combinatorics.

At the very least, the determination of the asymptotic behaviour of $\Ryser(r,t)$ when $t$ is linear in $r$ is an intriguing question in its own right, and even just reducing the grey area in Figure~\ref{fig:asymptotics} seems to require new ideas. This further motivates the pursuit of Ryser-type problems for various other classes of hypergraphs commonly studied in the field, some of which we outline below.  We believe that the techniques and constructions used in answering these questions could shed further light on Conjecture~\ref{conj:strong} and perhaps even on Ryser's Conjecture.\footnote{For instance, our proof that $\Ryser(6,2) = 3$ reduced the problem to the strictly $2$-intersecting case, and our proof that $\Ryser(7,2) \le 4$ used a double counting argument inspired by the proof of Lemma \ref{lem:rtd_counting} below.}

\subsection{Strictly $t$-intersecting hypergraphs} \label{sec:strictly}

One advantage of the construction of Bustamante and Stein~\cite{bustastein}, in which each vertex of the truncated projective plane is replaced by a set of $t$ vertices, is that it is regular. On the other hand, in our construction for Theorem~\ref{thm:lowerbound}, while the majority of vertices are in at most one edge, some vertices have very large degree. This begs the question of whether or not one can find a regular construction matching our bound, but, as we shall now show, the great irregularity is necessary for the cover number of the hypergraph to be large.

To start, observe that if $\mc H$ is a $d$-regular $(r,t)$-graph, and $V_i$ is any one of the $r$ parts, then $\mc H$ has exactly $d \card{V_i}$ edges, since each edge meets $V_i$ in exactly one vertex. Since any set $S \subset V(\mc H)$ can cover at most $d \card{S}$ edges, it follows that $\tau(\mc H) \ge \card{V_i}$; that is, the parts are minimum covers. Therefore, maximising the cover number of $d$-regular $(r,t)$-graphs is equivalent to maximising the number of vertices in such graphs.

An upper bound was provided by Frankl and F\"uredi~\cite{frafur86}, with a short proof later given by Calderbank~\cite{calderbank87}: they proved that any regular $t$-intersecting $r$-uniform hypergraph can have at most $(r^2 - r + t)/t$ vertices. In the $r$-partite setting, it follows that we have a part of size at most $(r-1)/t + 1/r < r/t$, and hence this is an upper bound on the cover number of any regular $(r,t)$-graph. Note that for $t \ge 3$ this is significantly smaller than the lower bound of Theorem~\ref{thm:lowerbound}, showing that the added condition of regularity considerably restricts the cover number of $(r,t)$-graphs.

Frankl and F\"uredi~\cite{frafur86} and Calderbank~\cite{calderbank87} further showed that the hypergraphs achieving equality in their bound are precisely the symmetric $2$-$(v,r,t)$ designs, a class of hypergraphs we now define.

\begin{dfn} \label{def:design}
Given $v, r, t \in \mathbb{N}$, a \emph{$2$-$(v,r,t)$ design} is an $r$-uniform hypergraph on $v$ vertices with the property that any two vertices share exactly $t$ common edges. The design is \emph{symmetric} if it has exactly $v$ edges.
\end{dfn}

Note that designs are never $r$-partite, since two vertices in the same part could not have any common edges. One might therefore hope for an even smaller upper bound if the hypergraph is also $r$-partite, but the construction of Bustamante and Stein shows that there can be regular $(r,t)$-graphs with cover number $r/t - 1$, and so there is not much room for improvement in general.\footnote{Using the truncated projective plane for some prime power $q$, the Bustamante-Stein construction gives $d$-regular $(r,t)$-graphs with cover number $\tau(\cH)=d=r/t-1=q$. For other values of the parameters $d, r$ and $t$, it may be possible to obtain better upper bounds, for instance in Corollary \ref{cor:regular}, which gives a stronger bound when $d<r/t-1$. }

Still, when it comes to $(r,t)$-graphs, our next result shows that one can obtain strong upper bounds on the cover number even if the condition of regularity is weakened to just having some control over the minimum and maximum degrees.

\begin{lem}
\label{lem:rtd_counting}
Let $\Delta$ be the maximum degree and $\delta$ the minimum degree of an $(r, t)$-graph $\m H$. 
Then 
\[\tau(\m H) \leq \left(\frac{\Delta - 1}{\delta}\right) \frac{r}{t} - \frac{\Delta - \delta - 1}{\delta}.\]
\end{lem}
\begin{proof}
Let $m$ be the total number of edges in $\mc H$, and let $e$ be one such edge. Double-counting pairs $(v, f)$ where $f \in E(\mc H) \setminus \{e\}$ and $v \in e \cap f$, we get $r(\Delta - 1) \geq (m - 1)t$, or
\begin{equation} \label{eqn:rtd_dblcount}
\left( \Delta - 1 \right) \frac{r}{t} + 1 \ge m .
\end{equation}
Let $u$ be a vertex of maximum degree $\Delta$ and let $P$ be the part of the $r$-partition that contains $u$. Since every edge is incident to a unique vertex in $P$, by looking at the edges incident to each vertex in $P$ we get 
\begin{equation} \label{eqn:rtd_partbound}
m \geq (|P| - 1) \delta + \Delta \ge ( \tau( \m H) - 1) \delta + \Delta,
\end{equation}
where the final inequality follows from the fact that $P$ is a vertex cover. Combining the upper and lower bounds on $m$ then gives the desired result. 
\end{proof}

In particular, this restricts the cover number of $d$-regular $(r,t)$-graphs, and, if $d < r$, the bound we obtain is smaller than that derived from Frankl and F\"uredi~\cite{frafur86} and Calderbank~\cite{calderbank87}. As with their results, we can characterise the hypergraphs achieving equality, for which we require a couple more design-theoretic definitions.

\begin{dfn} \label{def:dual_resolvable}
Given a hypergraph $\m H$, the \emph{dual hypergraph} $\m H^D$ has $V(\m H^D) = E(\m H)$ and \[E(\m H^D) =~\{ \{ e~\in E(\m H) : u \in e \} : u \in V(\m H) \};\] that is, we transpose the incidence relation between vertices and edges. Also, we say a $2$-$(v,r,t)$ design is \emph{resolvable} if its edges can be partitioned into perfect matchings.
\end{dfn}

Now we can state our result for the regular setting.

\begin{cor} \label{cor:regular}
If $\m H$ is a $d$-regular $(r,t)$-graph, then
\[ \tau(\m H) \le \frac{r}{t} - \frac{r}{dt} + \frac{1}{d}, \]
with equality if and only if $\m H$ is the dual of a resolvable $2$-$(v,d,t)$ design.
\end{cor}

For example, for a prime power $q$ and dimensions $1 \le k < n$, the $k$-dimensional affine subspaces in $\mathbb{F}_q^n$ form a resolvable $2$-$(q^n, q^k, \binom{n-1}{k-1}_q)$ design,\footnote{Where $\binom{n-1}{k-1}_q = \frac{(q^{n-1}-1)(q^{n-2} - 1) \hdots (q^{n-k+1} - 1)}{(q^{k-1} - 1)(q^{k-2)-1} \hdots (q-1)}$ is the Gaussian binomial coefficient, which counts the number of $k$-dimensional spaces that contain a given pair of points.} whose dual is therefore a tight construction for Corollary~\ref{cor:regular}.\footnote{When $n = 2$, the dual is the truncated projective plane, and thus this construction generalises the classic tight construction for Ryser's Conjecture.} In fact, we have a rich and storied variety of extremal constructions, as the study of resolvable designs dates back to Kirkman's famous schoolgirl problem~\cite{kirkman} from 1857, which asked for resolvable $2$-$(15,3,1)$ designs. This was greatly generalised by Ray-Chaudhuri and Wilson~\cite{RCW1, RCW2}, who showed the existence of resolvable designs of all uniformities $r$ whenever $v$ is sufficiently large and the trivial divisibility conditions are satisfied. More recently, Keevash~\cite{keevash} resolved some long-standing conjectures by extending these results to designs of greater strength, while results of Ferber and Kwan~\cite{Ferber-Kwan19} suggest that, when $v \equiv 3 \pmod 6$, almost all of the exponentially (in $v^2$) many $2$-$(v,3,1)$ designs should be resolvable.

\begin{proof}[Proof of Corollary~\ref{cor:regular}]
The upper bound follows immediately from Lemma~\ref{lem:rtd_counting} by substituting $\Delta = \delta = d$.

For the characterisation of equality, first observe that the inequality~\eqref{eqn:rtd_dblcount} is always tight if and only if any two edges of $\m H$ share exactly $t$ vertices, in which case we say $\m H$ is strictly $t$-intersecting. Next, we note that the inequalities in~\eqref{eqn:rtd_partbound} are always tight in the regular setting; the first because all degrees are equal to $d$, and the second because, as argued at the beginning of Section~\ref{sec:strictly},  a part is always a minimum cover in a regular $r$-partite $r$-uniform hypergraph. Thus we see that we have equality if and only if $\m H$ is strictly $t$-intersecting.

Now suppose $\m H$ is $d$-regular $(r,t)$-graph that is strictly $t$-intersecting, and consider the dual hypergraph $\m H^D$. Since every vertex of $\m H$ has degree $d$, every edge of $\m H^D$ contains $d$ vertices. Furthermore, as every pair of edges in $\m H$ shares $t$ vertices, every pair of vertices in $\m H^D$ have $t$ common edges. Thus $\m H^D$ is a $2$-$(v,d,t)$ design, where $v$ is the number of edges in $\m H$. Finally, since each edge of $\m H$ contains exactly one vertex from any of the $r$ parts, a part corresponds to a perfect matching in $\m H^D$, with every vertex covered exactly once. Hence, since $\m H$ is $r$-partite, the edges of $\m H^D$ can be partitioned into $r$ perfect matchings; that is, $\m H^D$ is resolvable.

Conversely, the same reasoning shows that the dual $\m H$ of a resolvable $2$-$(v,d,t)$ design gives a strictly $t$-intersecting $d$-regular $r$-partite $r$-uniform hypergraph with $v$ edges, where $r$ is the number of perfect matchings in the resolution of the design. From our above remarks, this implies the dual achieves equality in the upper bound on $\tau(\m H)$.
\end{proof}

In the above proof, we saw that for a regular $(r,t)$-graph to have as large a cover number as possible, it must be strictly $t$-intersecting, with every pair of edges meeting in exactly $t$ vertices. It is therefore natural to ask what happens when we drop the condition of regularity, and only require the $(r,t)$-graph be strictly $t$-intersecting. Such study has previously been carried out in the setting of Ryser's Conjecture for ($1$-)intersecting families.

Recall that Ryser's Conjecture for intersecting $r$-partite hypergraphs was proved by Tuza~\cite{Tuza} for all $r \leq 5$. 
Franceti\'c, Herke, McKay and Wanless~\cite{Francetic17} showed that if we restrict ourselves to linear (that is, strictly $1$-intersecting) hypergraphs, then the conjecture is true for all $r \leq 9$. 
Inspired by this, we prove Conjecture~\ref{conj:weak} for strictly $t$-intersecting hypergraphs for a much wider range of parameters $r$ and $t$ than covered by Corollary~\ref{cor:comparisons}. 

\begin{thm} \label{thm:strictly}
Let $t \ge 1$ and $t < r \le t^2 + 3t - 1$ be integers. If $\mc H$ is a strictly $t$-intersecting $r$-partite hypergraph, then $\tau(\mc H) \leq r - t$. 
\end{thm}

\begin{proof}
For the sake of contradiction, suppose $\tau \geq r - t + 1$, and, for the sake of convenience, let $\delta = \delta(\mc H)$,  $\Delta = \Delta(\mc H)$ and $\tau = \tau(\mc H)$. 
 
Let $v$ be an arbitrary vertex of $\mc H$ and let $e$ be an edge through $v$. 
Since $\tau \geq r-t+1$, 
we have that for every set $S$ of $r - t$ vertices in $e \setminus \{v\}$, there exists an edge $f$ of $\mc H$ with $e \cap f \subseteq e \setminus S$. 
Since $\mc H$ is $t$-intersecting and $|e \setminus S| = t$, we must have $e \cap f = e \setminus S$, which in particular implies that $v \in f$. This shows that $d(v) \geq \binom{r - 1}{r - t} + 1$, and, since $v$ was arbitrary, we get $\delta \geq \binom{r - 1}{r - t} + 1$. 
Solving the inequality in Lemma~\ref{lem:rtd_counting} for $\Delta$, we obtain
\[\Delta \geq \frac{\delta(\tau - 1)}{r/t - 1} + 1,\]
which, upon substituting $\delta \geq \binom{r - 1}{r - t} + 1$ and $\tau \geq r - t + 1$, yields
\[\Delta \geq t\binom{r - 1}{r - t} + t + 1.\]

Now let $u$ be a vertex of degree $d(u) = \Delta$ and let $f_0$ be an edge through $u$. There are $\binom{r - 1}{t - 1} = \binom{r - 1}{r - t}$ choices of $t$-subsets of $f_0$ containing $u$, and every edge $f' \neq f_0$ through $u$ intersects $f_0$ in one of these sets. Since there are at least $t \binom{r - 1}{r - t} + t$ edges through $u$ other than $f_0$, by the pigeonhole principle there must exist a $t$-subset $S$ of $f_0$ containing $u$ for which there are at least $t + 1$ edges $f_1, \dots, f_{t + 1}$ with $f_0 \cap f_i = S$ for all $i$.  Since $\m H$ is strictly $t$-intersecting, we further have $f_i \cap f_j = S$ for all $i \neq j$. 

We claim that $S$ is a vertex cover. If not, there exists an edge $e$ such that $e \cap f_i \subseteq f_i \setminus S$ for all $0 \leq i \leq t + 1$. Since $f_0 \setminus S, \dots, f_{t + 1} \setminus S$ are disjoint sets, and $e$ can only contain one vertex from each part of the $r$-partition, we get $t(t + 2) \leq r - t$, contradicting our upper bound on $r$.

Thus, we have a vertex cover $S$ of size $t$. 
Since Conjecture~\ref{conj:weak} is known for $r \le 2t$, it follows that $\tau \le r - t$, contradicting our original supposition that $\tau \geq r - t + 1$. 
\end{proof}

Although the restriction of being strictly $t$-intersecting allows us to prove the bound from Conjecture~\ref{conj:weak} for a wider range of parameters $(r,t)$, we believe this is far from tight.  Indeed, in this setting, we even lack constructions that come close to the smaller bound of Theorem~\ref{thm:lowerbound}. The best constructions we have found thus far are the duals of resolvable designs, as given in Corollary~\ref{cor:regular}. As these are also regular, their cover numbers are smaller than $\frac{r}{t}$, significantly smaller than the upper bound of Theorem~\ref{thm:strictly}.

\begin{prob}
Prove that $\tau(\m H) \le \frac{r}{t}$ for any strictly $t$-intersecting $r$-partite $r$-uniform hypergraph $\m H$, or find constructions with larger cover numbers.
\end{prob}

\subsection{$s$-covers}
Another new direction is to ask for more from our vertex covers -- rather than just intersecting each edge, we could ask for a set that meets every edge in many vertices.

\begin{dfn}
Let $\m H$ be an $(r, t)$-graph.
For $s \geq 1$, we define an $s$-cover of $\m H$ to be a set $B \subseteq V(\m H)$ such that $\card{B \cap e} \ge s$ for every $e \in E(\m H)$.  We further define
\[\tau_s(\m H) = \min\{|B| : B \text{ is an } s\text{-cover of } \m H \}.\]
\end{dfn}

Observe that $\tau_1(\m H) = \tau(\m H)$. 
We can then generalise Ryser's Conjecture (in the intersecting case) by asking for the maximum of $\tau_s(\m H)$ over all $(r, t)$-graphs. If $s \le t$, then, since every pair of edges intersects in at least $t$ vertices, any edge $e \in E(\m H)$ is an $s$-cover, and so we always have $\tau_s(\m H) \le r$.

However, the problem is ill-posed if $s > t$.  For arbitrary $n \in \mathbb{N}$, we can take $\m H$ be the complete $r$-partite $r$-graph with parts $V_1, \hdots, V_r$, where $\card{V_1} = \hdots = \card{V_t} = 1$ and $\card{V_{t+1}} = \hdots =  \card{V_r} = n$.  It is then easy to see that $\tau_s(\m H) = (s-t)n + t$, and therefore $\tau_s(\m H)$ is unbounded for $(r,t)$-graphs. We shall thus require $1 \le s \le t \le r$.

\begin{dfn}
Given integers $1 \le s \le t \le r$, define
\[ \Ryser_s(r,t) = \max \{ \tau_s(\m H) : \m H \textrm{ is an } (r,t)\textrm{-graph} \}. \]
\end{dfn}

The case $s = 1$ is obviously what we have been talking about all along, and the following lemma shows how we can leverage our constructions from that case to obtain lower bounds when $s \ge 2$.

\begin{lem} \label{lem:scoverrecursive}
For all $1 \le s \le t \le r$ and $a \ge 1$,
\[ \Ryser_{s + a}(r + a, t + a) \ge \Ryser_{s}(r, t) + a. \]
\end{lem}
\begin{proof}
Let $\m H'$ be an $(r, t)$-graph with $\tau_{s}(\m H) = \Ryser_s(r,t)$.  
Form $\m H$ by adding the same set $S$ of $a$ vertices to each edge of $\m H'$.  
$\m H$ is then an $(r + a, t + a)$-graph. 
Let $B$ be a smallest $(s+a)$-cover of $\m H$.  By removing $a$ elements from $B$, including all members of $B \cap S$, we obtain an $s$-cover $B'$ of $\m H'$.  
Hence we must have $\card{B'} \ge \tau_{s}(\m H') = \Ryser_s(r,t)$, and thus $\Ryser_{s + a}(r + a, t + a) \ge \tau_{s+a}(\m H) = \card{B} \ge \Ryser_{s}(r, t) + a$.
\end{proof}

The next proposition extends Theorem~\ref{thm:exactvalue} to the case when $s \ge 2$.

\begin{prop} \label{prop:scover1}
If $r \le 3t - 2s$, then
\[ \Ryser_s(r,t) = \left\lfloor \frac{r-t}{2} \right\rfloor + s. \]
\end{prop}

\begin{proof}
The lower bound, valid for all $1 \le s \le t \le r$, is an easy consequence of Lemma~\ref{lem:scoverrecursive} and Theorem~\ref{thm:lowerbound}:
\[ \Ryser_{s}(r, t) \geq \Ryser_{1}(r - s + 1, t - s + 1) + s - 1 \ge \floor*{\frac{r-t}{2}} + s. \]

For the upper bound, let $t \ge \tfrac{r + 2s}{3}$.
Let $\m H$ be an $(r, t)$-graph, and let $t' \geq t$ be the minimum size of an intersection of two edges. 
 We will show $\tau_s(\m H) \le \tau' = \floor*{ \frac{r-t'}{2} } + s$.  Let $e_1$ and $e_2$ be two edges intersecting in exactly $t'$ elements, and let $S = e_1 \cap e_2$.  By our bound on $t$, we have $t' \ge \tau'$.

Let $B$ be a set of $\tau'$ elements from $S$.  We claim that $B$ is an $s$-cover.  Indeed, suppose there was some $e_3 \in E(\m H)$ with $\card{B \cap e_3} \le s-1$.  Then $e_3$ can contain at most $s-1 + t' - \tau'$ elements from $S$.  In all parts outside $S$, $e_3$ intersects $e_1 \cup e_2$ in at most one vertex.  Thus
\[ \card{e_1 \cap e_3} + \card{e_2 \cap e_3} \le 2(s-1 + t' - \tau') + r - t' = r + 2s + t' - 2\tau' - 2 \le 2t' - 1, \]
which contradicts $e_3$ intersecting both $e_1$ and $e_2$ in at least $t'$ elements.
\end{proof}

Observe that the lower bound was proved using Lemma~\ref{lem:scoverrecursive}, reducing the problem to the $s=1$ case via $\Ryser_s(r,t) \ge \Ryser(r-s+1, t-s+1)$.  When $s = t$, this also reduces the problem to the classic setting, where we have much stronger lower bounds. Here we know that, whenever $r-t$ is a prime power, $\Ryser(r-t+1,1) \ge r-t$, and thus $\Ryser_t(r,t) \ge r-1$. The constructions of Haxell and Scott~\cite{HS17} further show $\Ryser_t(r,t) \ge r-4$ for all sufficiently large $r$.

Our final result uses the construction of Bustamante and Stein~\cite{bustastein} to significantly improve the lower bound for large $r$, whenever $s > \frac{t}{2}$.

\begin{prop}\label{prop:scover2}
Let $1 \leq s \leq t$ and suppose $r = t(q + 1)$ for some prime power $q$.  Then
\[ \Ryser_s(r, t) \ge s \left( \frac{r}{t} - 1 \right). \]
\end{prop}
\begin{proof}
Let $\m H'$ be the truncated projective plane of order $q$, which is a $q$-regular $(q+1,1)$-graph with $q^2$ edges, and let $\m H$ be the hypergraph obtained by replacing each vertex of $\m H'$ with a set of $t$ vertices. $\m H$ is then a $q$-regular $(t(q+1), t)$-graph with $q^2$ edges.

Since each vertex covers $q$ edges, to cover all of the edges at least $s$ times, we require at least $\frac{sq^2}{q} = sq = s \left( \frac{r}{t} - 1 \right)$ vertices. Thus $\Ryser_s(r,t) \ge \tau_s(\m H) \ge s \left( \frac{r}{t} - 1 \right)$.
\end{proof}

It remains an open problem to find matching upper bounds in these ranges.

\begin{prob}
Determine $\Ryser_s(r,t)$, at least asymptotically, when $r$ is large and $1 < s < t$.
\end{prob}

\paragraph{Note added in proof} During the publication of our manuscript, it was brought to our attention that the $t=1$ case of Theorem~\ref{thm:kwise} (when $k \ge 3$) was earlier proven by Kir\'aly~\cite{kirsolo}, in the context of monochromatic components in edge-coloured hypergraphs. As observed by DeBiasio (personal communication), Kir\'aly's theorem can be used as a base case for induction on $t$, providing an alternative proof of Theorem~\ref{thm:kwise} for $k \ge 3$. For more information on various generalisations of Ryser's Conjecture, we refer the reader to the recent survey of DeBiasio, Kamel, McCourt and Sheats~\cite{DKMS20}.

\end{document}